\documentclass[paper,10pt]{article}
\usepackage{fullpage}
\usepackage{latexsym}
\usepackage{amsfonts, amsthm}
\usepackage[all]{xy}
\usepackage{rotating}
\usepackage{amsmath}
\usepackage{amssymb}
\usepackage{amscd}
\usepackage{mathrsfs}
\usepackage[english]{babel}
\usepackage{makeidx}
\usepackage{hyperref}
\usepackage[latin1]{inputenc}
\usepackage{fancyhdr}
\newtheorem{defi}{Definition}
\newtheorem{prop}[defi]{Proposition} 
\newtheorem{lemma}[defi]{Lemma} 
\newtheorem{teo}[defi]{Theorem}
 
\newtheorem{rmk}[defi]{Remark}

\newtheorem*{teo*}{Theorem}

\newcommand{\0}{\mathcal{O}}    
\newcommand{\Spf}{\mathrm{Spf}}

\newcommand{\Sp}{\mathrm{Spec}}

\newcommand{\cO}{\mathcal{O}}
\newcommand{\lra}{\longrightarrow}

\begin{document}
\title{\textsc{On a $p$-adic invariant cycles theorem}}
\author{ B. Chiarellotto \and R. Coleman \and V. Di Proietto \and A. Iovita}

\maketitle

\textsc{Abstract -} \begin{small} For a proper semistable curve $X$ over a DVR of mixed characteristics we reprove 
the "invariant cycles theorem'' with trivial coefficients (see  \cite{Ch99}) i.e.
that the group of elements annihilated by the  monodromy operator on the  first de Rham cohomology group of the generic 
fiber of 
$X$ coincides with the 
first rigid cohomology group of its special fiber, without the hypothesis that the residue field of $\cal V$ is finite.   
This is done 
using the explicit description of the monodromy operator on the de Rham cohomology of the generic fiber of $X$ 
with coefficients convergent $F$-isocrystals given  in \cite{CoIo10}. We apply these  
ideas to the case  where the coefficients are unipotent convergent   $F$-isocrystals defined on the special fiber (without 
log-structure): we show that the invariant cycles theorem does not  hold in general in this setting. Moreover we 
give a sufficient condition for the non exactness.
\end{small}
\bigskip
\section{Introduction}
\label{sec:intro}
Let $\cal V$ be a complete discrete valuation ring of mixed characteristics,  $K$  its fraction field and  $k$ 
the 
residue field, which we assume to be perfect. Let $W:=W(k)$ denote the ring of Witt-vectors with 
coefficients in $k$ seen as a subring of ${\cal V} $ and let $K_0$ denote its fraction field. 

For a proper variety $X$ over $\cal V$ with semistable reduction and 
special fiber 
$X_k$, via the  theory of log schemes and the  work of Hyodo-Kato one defines  a monodromy operator  on the de Rham comology 
groups of its generic fiber $X_K$. It has been known for some time now that associated to this operator there is an 
analogue of the classical invariant cycles  sequence \cite{Ch99}
$$
H^i_{\mathrm{rig}}(X_k)\otimes_{K_0} K \rightarrow H^i_{\mathrm{dR}}(X_K) \rightarrow   H^i_{\mathrm{dR}}(X_K).
$$
The exactness of such a sequence is implied by the weight-monodromy conjecture  \cite{Ch99} if the residue field $k$ is finite. 
Hence the above invariant cycles  sequence is exact if    $X$ is a curve or a surface (which are the cases in which the 
weight-monodromy conjecture is known) and in this case the first map is 
even  injective if $i=1$.  I.e. the following sequence is exact:

\begin{equation} \label{CM}
0 \rightarrow H^1_{\mathrm{rig}}(X_k)\otimes_{K_0}K \rightarrow H^1_{\mathrm{dR}}(X_K) \rightarrow   
H^1_{\mathrm{dR}}(X_K).
\end{equation}
In these cases (i.e. in the cases in which the sequence (\ref{CM}) is exact)
we obtain an interpretation  of the part of the de Rham cohomology which is  annihilated by  the  monodromy 
operator: it is the rigid cohomology group of the special fiber.  On the other hand the same exact sequence gives 
us an interpretation \`a la Fontaine of the 
first rigid cohomology group, in fact we can translate the exactness  as 
follows: since  
$$D_\mathrm{st}(H^1_\text{\rm \'et}(X_K\times \overline{K}),\mathbb{Q}_p)=H^1
_\mathrm{log-crys}(X_k)\otimes K$$
$$D_\mathrm{st}^{N=0}=D_\mathrm{crys}, $$
then 
$$H^1_\mathrm{rig}(X_k)=D_\mathrm{crys}(H^1_\text{\rm \'et}(X_K\times \overline{K}),\mathbb{Q}_p)$$ 

\bigskip
\noindent
In \cite{CoIo10} another definition  of a monodromy operator  was given in the case $X$ is a 
curve with semistable reduction using the combinatorics of the curve together with the use of the analytic spaces 
associated to the generic fiber. 
The authors also considered the case of cohomology with coefficients and generalized the definition of the monodromy 
operator 
on the de Rham cohomology with coefficients  non trivial log-$F$- isocrystals and they showed that  it coincides with the 
previous definition given by Faltings \cite{Fa}.  Using this definition of the monodromy operator we are able (see \S5) to 
re-prove the exactness of the invariant cycles sequence (\ref{CM}) without any hypothesis on the finiteness of the 
residue field.  
It is then natural to ask if such an invariant cycles sequence (\ref{CM}) is still exact  when the log-$F$-isocrystals are induced from  convergent $F$-isocrystals on the special fiber.   This is one of themotivations of the present article. 
As a matter of fact,   
the invariant cycles sequence (\ref{CM})  involves the trivial  convergent $F$-isocrystal  on the special fiber of 
$X$ and its 
rigid cohomology.
Hence we start with  coefficients which a priori do not have singularities being convergent on the 
special fiber without any log structure.  But, even for the  
simplest non-trivial coefficients on a curve  (i.e. 
the unipotent ones) the sequence   fails sometimes to be exact and we give  a sufficient condition 
(see Theorem \ref{VIP}). Underlying our work, of course, is the aim of giving a  cohomological interpretation 
for the part of the cohomology on which the monodromy operator acts as zero.

Of course the invariant cycles theorem can be studied also in the $\ell$-adic
and respectively the complex settings, where it is known for semi-simple perverse sheaves or ${\cal D}$-modules
of geometric origin and it follows from the decomposition theorem (\cite{BDD} corollaire 6.2.8, \cite{Mo} 
theorem 19.47,
\cite{Sa}, theorem 1, \cite{MiDe}, theorem section 1.7).
Our $p$-adic setting deals with unipotent, non-trivial coefficients, which are therefore not semi-simple.
We did not find any evidence of a similar result for reducible coefficients in the $\ell$-adic or complex 
settings, although we believe that such results should hold. 

Here  it is the plan of our article. In \S2 we introduce   notation and recall results on rigid spaces which will be used 
in the article, in the third  paragraph we  recall  some properties of the monodromy operator  on the de Rham 
cohomology with 
coefficients on a curve 
as introduced by Coleman and Iovita and of the associated invariant cycles sequence.  In \S4  we  give some 
properties of such a 
monodromy operator: in particular for general  convergent $F$-isocrystals  we  prove  that the  rigid cohomology of the  
convergent $F$-isocrystal injects on the  part of the de Rham cohomology  of the associated  log-$F$ isocrystal 
where the monodromy 
acts as zero.  In \S5, we then  re-prove  (\cite{Ch99})  the invariant cycles theorem for trivial coefficients 
in a combinatorial way 
along the lines of the work in  \cite{CoIo10}.  In  \S6 we study the  invariant cycles sequence   for unipotent  
convergent  
$F$-isocrystals  and  we prove a sufficient conditions for the non exactness of the sequence.  
Finally we give an explicit example of this on a Tate curve. 

$$ \ast \ast$$
We thank Benoit Larose for suggestions and helpful discussions on graph theory. We also thank Thomas Zink for 
pointing out to 
one of us that it would be interesting to study a Fontaine type interpretation of the rigid cohomology of a 
smooth proper variety in characteristic $p$. Thanks are also due to Claude Sabbah for interesting 
email exchanges on questions pertaining to this research. The first author was supported by the Cariparo Eccellenza Grant 
"Differential Methods in Arithmetic, Geometry and Algebra". The third author was was supported by a postdoctoral fellowship and kaken-hi (grant-in-aid) of the Japanese Society for the Promotion of Science (JSPS) while working on this paper at the University of Tokyo.

%
%

\section{Notation and Settings.  A Mayer-Vietoris exact sequence}\label{A Mayer-Vietoris exact sequence}


We assume the notations in section \ref{sec:intro}.
Let $X$ be a proper curve over $\cal V$ (of mixed characteristic) that is semistable, which means that locally for the Zariski topology there is an \'etale map to $\Sp(\mathcal{V}[x,y]/xy-\pi)$ and we suppose that the special fiber, union of smooth irreducible components, has at 
least two components. 
We denote by $X_k$ the special fiber of $X$ which we suppose connected, by $X_K$ its generic fiber and by 
$X^{\mathrm{rig}}_K$ the 
rigid analytic generic fiber. By theorem 2.8 in \cite{lichtenbaum} $X$ being a proper, regular curve over $\cal{V}$ is in fact
a projective $\cal{V}$-scheme.\\

Following \cite{CoIo99} we  associate to $X_k$ a graph $Gr(X_k)$ whose definition we now recall. 
To every irreducible component $C_v$ of $X_k$ we associate a vertex $v$ and if $v,w$ 
are vertices, an oriented 
edge $e=[v,w]$ with origin $v$ and end $w$  corresponds to an intersection point $C_e$ of the components $C_v$ and $C_w$. 
We denote by $\mathscr{V}$ the set 
of vertices and by $\mathscr{E}$ the set of oriented edges. \\
We have the specialization map  
$$\mathrm{sp}:X_K^{\mathrm{rig}}\rightarrow X_k$$
defined  in \cite{Be}.\\
For every $v \in \mathscr{V}$ we define 
$$X_v:=sp^{-1}(C_v)$$
and for every $e \in\mathscr{E}$
$$X_e:=sp^{-1}(C_e).$$
The set $X_e$ is an open annulus in $X_K^{\mathrm{rig}}$ and $X_v$ is what is called a wide open subspace in
(\cite{Co89} proposition 3.3), that means an open of $X^{\mathrm{rig}}_K$ isomorphic to 
the complement of a finite number of closed disks, each contained in a residue class, in a smooth proper curve 
over $K$ with good reduction. If $C_v$ and $C_w$ intersect in $C_e$, then $X_v\cap X_w=X_e.$ \\
One can prove that $\{X_v\}_{v\in \mathscr{V}}$ is an admissible covering of $X^{\mathrm{rig}}_K$ (\cite{Co89}) and that wide 
opens 
are Stein spaces so that we can use the covering $\{X_v\}_{v\in \mathscr{V}}$ to calculate the de Rham cohomology of 
$X_K^{\mathrm{rig}}$ using a \v{C}ech complex.
Moreover one can prove that the first de Rham cohomology of a wide open is finite (\cite{Co89} theorem 4.2) proving a comparison theorem with the de Rham cohomology of an algebraic curve minus a finite set of points. 
Let $(\mathcal{E}, \nabla)$ be a module with integrable connection on $X_K^{\mathrm{rig}}$.\\
Given the admissible covering $\{X_v\}_{v\in \mathscr{V}}$ that is such that its elements intersect only two by two, we can write the Mayer-Vietoris sequence:
\begin{equation}\label{M-V}
\xymatrix{
\oplus_{v\in \mathscr{V}}H^0_{\mathrm{dR}}(X_v, (\mathcal{E}, \nabla))\ar[r]^{\alpha}&\oplus_{e\in\mathscr{E}}H^0_{\mathrm{dR}}(X_e, (\mathcal{E}, \nabla))\ar[r] &H^1_{\mathrm{dR}}(X_K^{\mathrm{rig}}, (\mathcal{E}, \nabla))\ar `r[d] `[l] `[lld] `[dl] [dl] \\ & \oplus_{v\in \mathscr{V}}H^1_{\mathrm{dR}}(X_v, (\mathcal{E}, \nabla))\ar[r]^{\beta}&\oplus_{e\in\mathscr{E}}H^1_{\mathrm{dR}}(X_e, (\mathcal{E}, \nabla)).
}
\end{equation}

Let us remark that every cohomology group that appears in the long exact sequence except for $H^1_{\mathrm{dR}}(X^{\mathrm{rig}}_K, (\mathcal{E}, \nabla))$ can be calculated as the cohomology of the global sections of the de Rham complex, due to the fact that every wide 
open is Stein.\\
From the equation (\ref{M-V}) we can deduce the short exact sequence
\begin{equation}\label{M-V2}
\xymatrix{
0\ar[r]& H^1(Gr(X_k), \mathcal{E})\ar[r]^{\gamma} &H^1_{\mathrm{dR}}(X_K^{\mathrm{rig}}, (\mathcal{E}, \nabla)))\ar[r]&\mathrm{Ker}(\beta)\ar[r]&0.
}
\end{equation}
where $H^1(Gr(X_k), \mathcal{E}):=\mathrm{Coker}(\alpha)$.

 \section{The monodromy operator and the rigid cohomology}\label{Monodromy operator and rigid cohomology}

We consider again a proper and semistable curve $X$, its generic fiber  $X_K$ and its associated  rigid space 
$X_K^{\mathrm{rig}}$.  We 
recall the construction of the monodromy operator in  \cite{CoIo10} section 2.2. \\
By  our assumptions there is a proper scheme $P$ over $W$, smooth around $X_k$ and such that
we have a global embedding $X\hookrightarrow P\times_{\rm Spec(W)}{\rm Spec}({\cal V})=P_{\cal V}$.  
Let us denote by 
$P_k$ its special fiber and by $P_{K_0}^{\mathrm{rig}}$ and $P_{K}^{\mathrm{rig}}$ the rigid analytic spaces associated to 
$P$ and $P_ {\cal V}$; then one has the
following diagram:
\begin{equation*} 
\xymatrix{ 
 & &P_{K_0} ^{\mathrm{rig}}\ar[d]\ar[dl]_{\mathrm{sp}_P} \\
X_k \ar[r] &P_k \ar[r]&P}
 \end{equation*}
where the map between $P_{K_0} ^{\mathrm{rig}}$ and $P_k$ is the specialization map that we denote by 
$\mathrm{sp}_P$. We also have a specialization map $\mathrm{sp}_{P_{\cal V} }:P_K^{\mathrm{rig} } \lra P_k$.  One can 
consider the tubes $\mathrm{sp}_P^{-1}(X_k):=  ]X_k[_P$ and   $Y_K:=\mathrm{sp}_{P_{\cal V} } ^{-1}(X_k)=]X_k[_{P_{\cal V} } $. 
Let now $E$ be a convergent $F$-isocrystal on $X_k$.  
It has a realization on 
$]X_k[_P$: $( \mathcal{E}, \nabla)$ and we denote by 
$(\mathcal{E}, \nabla)_K$ its base change to $K$. 
It is a module with connection on $Y_K$.
We will denote by the same symbol its restriction to $X_K^{\mathrm{rig}}$.  We may then 
define   the first rigid cohomology group with coefficients in $E$ as 
$$H^1_{\mathrm{rig}}(X_k, {E}):=H^1_{\mathrm{dR}}(]X_k[_P, (\mathcal{E}, \nabla)),$$ which is a finite dimensional 
$K_0$-vector space. We also consider 
$$
H^1_{\mathrm{rig}}(X_k, {E})_K:=H^1_{\mathrm{dR}}(]X_k[_{P_{\cal V} } , (\mathcal{E}, \nabla)_K)=H^1_{\mathrm{dR} }(Y_K, (\mathcal{E},
\nabla)_K).  
$$

On the other hand we can proceed as before and take $X_K^{\mathrm{rig}}$ as the rigid analytic space associated to $X_K$, 
we then have 
$$\varphi: X_K^{\mathrm{rig}}\longrightarrow Y_K$$ given by the immersion of $X$ into $P_{\cal V} $ that induces a map 
in cohomology 
\begin{equation} \label{IN}
\varphi^*: H^1_{\mathrm{rig}}(X_k, E)_K:=H^1_{\mathrm{dR}}(Y_K, (\mathcal{E}, \nabla)_K)\longrightarrow 
H^1_{\mathrm{dR}}(X^{\mathrm{rig}}_K, (\mathcal{E},\nabla)_K).
\end{equation}

In  the notations above  we define following \cite{CoIo10} a $K$-linear map $$N:H^1_{\mathrm{dR}}(X_K^{\mathrm{rig}}, 
(\mathcal{E}, \nabla)_K)\rightarrow H^1_{\mathrm{dR}}(X_K^{\mathrm{rig}}, (\mathcal{E},\nabla)_K).$$
Due to the fact that wide opens are Stein spaces, every element $[\omega]$ in $H^1_{\mathrm{dR}}(X_K, (\mathcal{E}, \nabla)_K)$ 
can be described as a hypercocycle $((\omega_v)_{v\in \mathscr{V}}, (f_e)_{e\in\mathscr{E}})$, with $(\omega_v)$ in $\Omega^1_{X_v}\otimes \mathcal{E}_{X_v} $ and $f_e$ in $\mathcal{E}_{X_e}$ that verifies that $\omega_{v|X_e}-\omega_{w|X_e}=\nabla (f_e)$ if $e=[v,w]$.\\
Let us remember that every $X_e$ is an ordered open annulus; we can define a residue map
$$\mathrm{Res}:H^1_{\mathrm{dR}}(X_e, (\mathcal{E}, \nabla)_K)\rightarrow H^0_{\mathrm{dR}}(X_e, (\mathcal{E}, \nabla)_K)$$
as follows. The module with connection $(\mathcal{E}, \nabla)_K$ has a basis of horizontal sections 
$e_1, \dots, e_n$ on $X_e$ because $X_e$ is a residue class (lemma 2.2 of \cite{CoIo10}). Hence if $z$ is an ordered 
uniformizer of the ordered 
annulus $X_e$ every differential form $\mu_e \in H^1_{\mathrm{dR}}(X_e, (\mathcal{E}, \nabla)_K)$ can be written as 
$\mu_e=\sum_{i=1}^n(e_i\otimes \sum_ja_{i,j} z^jdz)$ with $a_{i,j}\in K$. Then $\mathrm{Res}(\mu_e)=\sum_{i=1}^na_{i,-1}e_{i}$, 
and it is an isomorphism of vector spaces. 

For a cohomology class $[\omega]$ represented as before by $((\omega_v)_{v\in \mathscr{V}}, (f_e)_{e\in\mathscr{E}})$ we define 
$N$ as the composition of the following maps:
$$\tilde{N}: H^1_{\mathrm{dR}}(X^{\mathrm{rig}}_K, (\mathcal{E}, \nabla)_K)\longrightarrow  \oplus_{e\in\mathscr{E}}
H^0_{\mathrm{dR}}(X_e, (\mathcal{E}, \nabla)_K))$$ 
$$\tilde{N}:[\omega]\mapsto (\mathrm{Res}(\omega_{v|X_e})_{e=[v,w]})$$
and the map 

\begin{equation*}
\xymatrix{
i:\oplus_{e\in\mathscr{E}}H^0_{\mathrm{dR}}(X_e, (\mathcal{E}, \nabla)_K)\ar[r]& \oplus_{e\in\mathscr{E}}H^0_{\mathrm{dR}}
(X_e, (\mathcal{E}, \nabla)_K)/\oplus_{v\in \mathscr{V}}H^0_{\mathrm{dR}}(X_v, (\mathcal{E}, \nabla)_K)\ar[r]^{\,\,\,\,\,\,\,\;\;\;\;\;\;\;\;\;\;\;\;\;\;\;\;\;\;\;\;\;\;\;\gamma}&H^1_{\mathrm{dR}}(X_K^{\mathrm{rig}}, (\mathcal{E}, \nabla)_K) 
}
\end{equation*}
$$i:(f_{e})_{e\in\mathscr{E}}=\Big(0,f_e/\mathrm{Im}(\alpha)\Big)_{v \in \mathscr{V},e\in\mathscr{E}}$$
and $\gamma$ and $\alpha$ the same map as in (\ref{M-V}).\\
Hence $N$ is defined as $N=i \circ \tilde{N}$. Note that $N^2=0$.\\

\bigskip

In order to give an interpretation of the  monodromy operator on the de Rham cohomology defined above we'll introduce 
the log formalism.  The curve $X$ can be equipped with a log structure, associated to the special fiber $X_{k}$ which 
is a divisor 
with normal crossing and $\mathrm{Spec}(W)$ with the log structure given by the closed point. Pulling them back to 
$X_{k}$ and to $\mathrm{Spec}(k)$ respectively, we may consider $X_{k}$ and $\mathrm{Spec}(k)$ as log schemes, 
and when we want to 
treat them as log schemes we denote them by $X_{k}^{\times}$ and $\mathrm{Spec}(k)^{\times}.$ The log structure on 
$\mathrm{Spec}(W)$ induces a log structure on $\Spf( W)$, and again when we want to treat it as a log formal
scheme we denote 
it by  $\Spf(  W)^{\times}$.  We note  that  in the  case of the trivial isocrystal  by \cite{H-K} 
the de Rham cohomology 
groups of the generic fiber  coincide with the log-crystalline ones of $X_{k}^{\times}$, base-changed to $K$ .  
This result holds also in our case 
with coefficients. In fact if we start with a convergent $F$-isocrystal on  $X_k$, then one can associate to it a 
log-convergent $F$-isocrystal  on $X_{k}^{\times}$ and then a log(-crystalline) $F$-isocrystal on $X_{k}^{\times}$ 
((\cite{Sh1} theorem 5.3.1): we again denote it by $E$. 

\begin{prop}\label{propFalt}
In the previous hypothesis and notations  if we start with a convergent $F$-isocrystal $E$ on $X_k$ and we denote by 
$(\mathcal{E}, \nabla)$ its realization on $]X_k[_{P}$,  then the cohomology of the restriction  
$H^i_{\mathrm{dR}}(X^{\mathrm{rig}}_K, (\mathcal{E}, \nabla)_K)$ coincides with  the log-crystalline cohomology of the 
associated log-$F$-isocrystal on $X_{k}^{\times}$, $H^i_{\mathrm{log-crys}}(X_{k}^{\times}, E)\otimes_{K_0}  K$. The monodromy 
operators coincide as well. \end{prop}
\begin{proof}
We  are in the case of \cite{Fa}. The Frobenius structure will imply that the relative  log cohomology arising 
from the 
deformation gives a locally free module,  but it will guarantee also that   the exponents of the associated  
Gauss-Manin differential system  are non-Liouville numbers: hence we may trivialize the system by the transfer theorem
\cite{christol}.  
For the coincidence of the monodromy operators we refer to \cite{CoIo10}.
\end{proof}

Using $\varphi^*$  of (\ref{IN}) and  the monodromy operator $N$ we can form the following sequence
\begin{equation}\label{BC}
\xymatrix{
H^1_{\mathrm{dR}}(Y_K, (\mathcal{E}, \nabla)_K)\ar[r]^{\varphi^*}&H^1_{\mathrm{dR}}(X^{\mathrm{rig}}_K, (\mathcal{E}, \nabla)_K)\ar[r]^{N}&H^1_{\mathrm{dR}}(X^{\mathrm{rig}}_K, (\mathcal{E}, \nabla)_K).
}
\end{equation}
In \cite{Ch99} it is proven the following theorem when $k$ is finite and for varieties of dimensions 1 and 2  and 
$X_k$ projective (see also \cite{Na}).
\begin{teo}\label{teoBC}
In the sequence (\ref{BC}) if $E$ is  the trivial isocrystal, then the map $\varphi^*$ is injective and $\mathrm{Imm}( \varphi^*)=\mathrm{Ker} (N)$.
\end{teo}

In the next paragraph we are going to prove that if $E$ is not necessarily the trivial isocrystal, then in the 
sequence (\ref{BC}) the map $\varphi^*$ is injective and $\mathrm{Im} (\varphi^*)\subset \mathrm{Ker} (N)$. 
Moreover if $E$ is the trivial isocrystal  we  will give a new proof of theorem \ref{teoBC} using the explicit 
description of the monodromy operator as introduced before.

\bigskip

\begin{rmk} According to \cite{CoIo10}  for the definition of the  monodromy operator on the de Rham cohomology  we 
didn't need 
either the  
Frobenius structure or an isocrystal: we just needed a connection on the generic fiber. In general  we don't know the interpretation of 
such an operator in terms of the integral structures.
\end{rmk}

\bigskip
\section{The behavior of the  monodromy operator}

We would like to study the properties of the monodromy operator as defined in the previous section and, in particular, 
the 
exactness of the sequence (\ref{BC}).  

As in section \ref{A Mayer-Vietoris exact sequence} let us consider the graph $Gr(X_k)$ associated to $X_k$, with vertices in $\mathscr{V}$ and edges in $\mathscr{E}$. For $v\in \mathscr{V}$ we denote by $X_v:=\mathrm{sp}_X^{-1}(C_v)$ and by $Y_v:=\mathrm{sp}_{P_{\cal{V}}}^{-1}(C_v)$; because the definition of $\varphi$, we have that $\varphi({X_v})\subset Y_v$. In the same way we denote by $X_e:=\mathrm{sp}_X^{-1}(C_e)$ and by $Y_e:=\mathrm{sp}_{P_{\cal{V}}}^{-1}(C_e)$; because the definition of $\varphi$, we have that $\varphi({X_e})\subset Y_e$. \\
Let us note that $Y_e$ is a polidisk because $P_{\mathcal{V}}$ is smooth. We choose the admissible covering of $X^{\mathrm{rig}}_K$ given by $\{X_v\}_{v\in\mathscr{V}}$ to calculate the de Rham cohomology using \v{C}ech complexes.\\
As before let $E$ be an $F$-convergent isocrystal on $X_k$, we can also use the Mayer-Vietoris spectral sequence for rigid cohomology 
with coefficients in $E$ (\cite{Tsu} theorem 7.1.2). We pick as finite close covering of $X_k$ the covering given by 
$\{C_v\}$. 
Since every intersection of three distinct components is empty the spectral sequence degenerates in a Mayer-Vietoris long exact 
sequence (\cite{Go} theorem 4.6.1)
\begin{equation}\label{M-Vyrig}
\xymatrix{
\oplus_{v\in \mathscr{V}}H^0_{\mathrm{rig}}(C_v, E)\ar[r]^{\alpha}&\oplus_{e\in\mathscr{E}}H^0_{\mathrm{rig}}(C_e, E)\ar[r] &H^1_{\mathrm{rig}}(X_k, E)\ar `r[d] `[l] `[lld] `[dl] [dl] \\ & \oplus_{v\in \mathscr{V}}H^1_{\mathrm{rig}}(C_v,E)\ar[r]^{\sigma}&\oplus_{e\in\mathscr{E}}H^1_{\mathrm{rig}}(C_e, E).
}
\end{equation} 
whose base-change to $K$ can be described  in terms of the de Rham cohomology of $Y_K$ as
\begin{equation}\label{M-Vy}
\xymatrix{
\oplus_{v\in \mathscr{V}}H^0_{\mathrm{dR}}(Y_v, (\mathcal{E}, \nabla)_K)\ar[r]^{\alpha}&\oplus_{e\in\mathscr{E}}H^0_{\mathrm{dR}}(Y_e, (\mathcal{E},\nabla)_K)\ar[r] &H^1_{\mathrm{dR}}(Y_K, (\mathcal{E},\nabla)_K)\ar `r[d] `[l] `[lld] `[dl] [dl] \\ & \oplus_{v\in \mathscr{V}}H^1_{\mathrm{dR}}(Y_v, (\mathcal{E}, \nabla)_K)\ar[r]^{\sigma}&\oplus_{e\in\mathscr{E}}H^1_{\mathrm{dR}}(Y_e, (\mathcal{E}, \nabla)_K).
}
\end{equation}

Now we study the exactness property of the sequence (\ref{BC}).
\begin{lemma}\label{injectivity}
If $E$ is a convergent isocrystal and $(\mathcal{E}, \nabla)$ is the coherent module with integrable connection induced by it, then 
the map $\varphi^*$ in the sequence (\ref{BC}) is injective.
\end{lemma}
\begin{proof}
We fix an irreducible component $C_v$ of $X_k$, we want to prove that the following sequence is exact:
\begin{equation}\label{G1}
0\longrightarrow H^1_{\mathrm{dR}}(Y_v, (\mathcal{E}, \nabla)_K)\longrightarrow H^1_{\mathrm{dR}}(X_v, (\mathcal{E}, \nabla)_K)\longrightarrow \bigoplus_{e\in \mathscr{E}_v}H^0_{\mathrm{dR}}(X_e, (\mathcal{E}, \nabla)_K);
\end{equation}
where the last map is the residue map and $\mathscr{E}_v:=\{e \mbox{ such that there exists a vertex }w \mbox{ with } 
e=[v,w]\}$.
As $C_v$ is proper and smooth the above sequence will be isomorphic to the following sequence:
\begin{equation}\label{G2}
0\longrightarrow H^1_{\mathrm{crys}}(C_v, E)\otimes K \longrightarrow H^1_{\mathrm{log-crys}}(C_v^{\times\times}, E) \otimes K\longrightarrow \bigoplus_{e \in \mathscr{E}_v}H^0_{\mathrm{dR}}(X_e, (\mathcal{E}, \nabla)_K),
\end{equation}
where $C_v^{\times \times}$ is the log scheme given by the component $C_v$ with the log structure induced by the divisor given by the intersection points of $C_v$ with the other components. The two sequence are isomorphic because $H^1_{\mathrm{crys}}(C_v, E)\otimes K\cong H^1_{\mathrm{dR}}(Y_v, (\mathcal{E}, \nabla)_K)$ since $C_v$ is proper and smooth, 
$H^1_{\mathrm{log-crys}}(C_v^{\times \times}, E)\otimes K\cong  H^1_{\mathrm{dR}}(X_v, (\mathcal{E}, \nabla)_K)$ by 
\cite{CoIo10} lemma 5.2. Moreover the second one is exact because is the Gysin sequence for rigid cohomology.\\
In fact the Gysin sequence for rigid cohomology is the following (proposition 2.1.4 of \cite{ChLeS}):
\begin{equation}\label{G3}
0\longrightarrow H^1_{\mathrm{rig}}(C_v, E)\otimes K\longrightarrow H^1_{\mathrm{rig}}(U_v, E)\otimes K\longrightarrow \bigoplus_{e\in \mathscr{E}_v}H^0_{\mathrm{dR}}(X_e, (\mathcal{E}, \nabla)_K),
\end{equation}
where $U_v$ is the complement in $C_v$  of all the points of $C_v$ that intersect the other components of $X_k$.\\
The isomorphism $H^1_{\mathrm{rig}}(U_v, E)\cong H^1_{\mathrm{log-crys}}(C_v^{\times\times}, E)$ follows from \cite{Sh02} 
paragraph 2.4 and \cite{Sh02} theorem 3.1.1. Moreover $H^0_{\mathrm{dR}}(X_e, (\mathcal{E}, \nabla)_K)\cong 
H^0_{\mathrm{dR}}(Y_e, (\mathcal{E}, \nabla)_K)$ because $Y_e$ and $X_e$ are residue classes and $E$ has a basis of 
horizontal sections on each residue class, which means that both $H^0_{\mathrm{dR}}(X_e, (\mathcal{E}, \nabla)_K)$ and 
$H^0_{\mathrm{dR}}(Y_e, (\mathcal{E}, \nabla)_K)$ are isomorphic to $K^d$ where $d$ is the rank of $\mathcal{E}$ as 
$\mathcal{O}_{X_K}$-module. 
Moreover by the Gysin isomorphism in degree zero  (proposition 2.1.4 of \cite{ChLeS}), with the same notations
as before, we have 
$H^0_{\mathrm{rig}}(C_v, E)\cong H^0_{\mathrm{rig}}(U_v, E)$, which implies that $H^0_{\mathrm{dR}}(X_v, (\mathcal{E}, \nabla)_K)
\cong H^0_{\mathrm{dR}}(Y_v, (\mathcal{E}, \nabla)_K)$, using the same techniques as before.\\ 
Using the Mayer-Vietoris long exact sequence for rigid cohomology (\ref{M-Vy}), we can pass to the following short 
exact sequence
\begin{equation}\label{M-Vy2}
\xymatrix{
0\ar[r]& H^1(Gr(X_k), \mathcal{E}_K)\ar[r]^{\delta} &H^1_{\mathrm{dR}}(Y_K, (\mathcal{E}, \nabla)_K))\ar[r]&\mathrm{Ker}(\sigma)\ar[r]&0.
}
\end{equation}
where $H^1(Gr(X_k), \mathcal{E}_K):=\mathrm{Coker}(\alpha)$.\\
Putting together Mayer-Vietoris sequences for the coverings $\{X_v\}$ and $\{Y_v\}$ respectively
we obtain the following diagram
\begin{equation}\label{bigd} 
\xymatrix{ 
        &                                & \oplus_{e}H^0_{\mathrm{dR}}(X_e, (\mathcal{E}, \nabla)_K)\ar[r]\ar[dl]_{\theta}&     \oplus_{e}H^0_{\mathrm{dR}}(X_e,(\mathcal{E}, \nabla)_K)&\\
0\ar[r]& H^{1}(Gr(X_k), \mathcal{E}_K)\ar[r]^{\gamma}&H^{1}_{\mathrm{dR}}(X_K, (\mathcal{E}, \nabla)_K) \ar[r]^{\pi_X}\ar[u]^{Res}&\oplus_{v}H^{1}_{\mathrm{dR}}(X_v, (\mathcal{E}, \nabla)_K) \ar[u]&\\
0\ar[r]& H^{1}(Gr(X_k), \mathcal{E}_K)\ar[r]^{\delta}\ar[u] &H^{1}_{\mathrm{dR}}(Y_K, (\mathcal{E}, \nabla)_K) \ar[u]^{\varphi^*} \ar[r]^{\pi_Y}&\oplus_{v}H^{1}_{\mathrm{dR}}(Y_v, (\mathcal{E}, \nabla)_K)\ar[u]^{\varphi^*}\ar[r]&0\\
          &                                                             &                                                                                    & 0\ar[u]}
 \end{equation}
and by the snake lemma one can conclude that $\varphi^{*}:H^{1}_{\mathrm{dR}}(Y_K, (\mathcal{E}, \nabla)_K)\rightarrow H^{1}_{\mathrm{dR}}(X_K, (\mathcal{E}, \nabla)_K)$ is injective.

\end{proof}
\begin{rmk}
Let us note that in (\ref{bigd}) the monodromy operator on $H^1_{\mathrm{dR}}(X_K, (\mathcal{E}, \nabla)_K)$ acts as $N=\gamma \circ \theta\circ Res $.\\

\end{rmk}
\begin{lemma}\label{Nphi}
If $E$ is a convergent $F$-isocrystal and $(\mathcal{E}, \nabla)$ is the coherent module with integrable connection 
induced by it, then in the sequence (\ref{BC})
$$N\circ \varphi^*=0.$$
\end{lemma}
\begin{proof}
Let us consider $[\omega]$ $\in$ $H^1_{\mathrm{dR}}(Y_K, (\mathcal{E}, \nabla)_K)$. Then $\varphi^*[\omega]$, which is an element of $H^1_{\mathrm{dR}}(X_K, (\mathcal{E}, \nabla)_K)$, can be represented by an hypercocycle 
$((\alpha_v)_{v\in \mathscr{V}}, (g_e)_{e\in \mathscr{E}})$ where $\alpha_v \in \Omega^1_{X_v}\otimes \mathcal{E}_{X_v} $ and $g_e$ in $\mathcal{E}_{X_e}$ and they verify that $\alpha_{v|X_e}-\alpha_{w|X_e}=\nabla (g_e)$ if $e=[v,w]$. We want to calculate $N(\varphi^*([\omega]))$. We now look at the diagram (\ref{bigd}). By the definition of $N$ one can see that 
$$N(\varphi^*([\omega]))=\gamma \circ \theta\circ Res (\varphi^*([\omega]))=\gamma \circ \theta\circ Res_{|X_e} (\pi_X(\varphi^*[\omega])).$$
By the commutativity of the diagram (\ref{bigd}) $Res_{|X_e} (\pi_X(\varphi^*[\omega]))=Res_{|X_e} (\varphi^*(\pi_Y([\omega])).$\\
If we denote by $\omega_v=\pi_Y([\omega])$, then we have to compute $Res_{|X_e}(\varphi^*(\omega_v))$:
$$Res_{|X_e}(\varphi^*(\omega_v))=Res (\varphi^*(\omega_v)_{|X_e})=Res(\varphi^*(\gamma_e))$$  
where $\gamma_e \in \mathcal{E}_{Y_e}\otimes\Omega^1_{Y_e}$, but as $Y_e$ is an open polydisc we have that
 $H^1_{\mathrm{dR}}(Y_e, (\mathcal{E}, \nabla)_K)=0$ and so $Res(\phi^*(\gamma_e))=0$ as claimed.
\end{proof}

From the above lemma we can conclude that $\mathrm{Im}(\varphi^*)\subset \mathrm{Ker} (N)$. Now we'd like  to 
characterize in terms of residues the elements of $H^1_{\mathrm{dR}}(X_K, (\mathcal{E}, \nabla)_K)$ which are in the 
image of $\varphi^*$.\\
Let us take $[\omega]$ $\in$ $H^1_{\mathrm{dR}}(X_K, (\mathcal{E}, \nabla)_K)$. As before we can choose a 
representative $\omega=((\omega_v)_{v\in \mathscr{V}}, (f_e)_{e\in\mathscr{E}})$, with $(\omega_v)$ in $\mathcal{E}_{X_e}\otimes\Omega^1_{X_v}$ and $f_e$ in $\mathcal{E}_{X_e}$ which verifies that $\omega_{v|X_e}-\omega_{w|X_e}=\nabla(f_e)$ if $e=[v,w]$.\\
In the next lemma we prove a necessary and sufficient condition for an element of $H^1_{\mathrm{dR}}(X_K,
(\mathcal{E}, \nabla)_K)$ to be in the image of the map $\varphi^*$.

\begin{lemma}\label{suffnec}
Let us take $[\omega]$ $\in$ $H^1_{\mathrm{dR}}(X_K, (\mathcal{E}, \nabla)_K)$ and a representative 
$\omega=((\omega_v)_{v\in \mathscr{V}}, (f_e)_{e\in\mathscr{E}})$ as above. 
Then $\mathrm{Res}_{X_e}(\omega_{v|X_e})=0$ for every $e\in\mathscr{E}$ if and only if $[\omega]\in \mathrm{Im}(\varphi^*)$.
\end{lemma}
\begin{proof}
Let us see first that if $\mathrm{Res}_{X_e}(\omega_{v|X_e})=0$ for every $e\in\mathscr{E}$, then $[\omega]\in \mathrm{Im}(\varphi^*).$
If $\mathrm{Res}_{X_e}(\omega_{v|X_e})=0$, then thanks to the exact sequence (\ref{G1}) there exists $\gamma_v \in H^1_{\mathrm{dR}}(Y_v, (\mathcal{E}, \nabla)_K)$ such that $\varphi^*(\gamma_v)=\omega_v$ for every $v \in \mathscr{V}$. As the map 
$\pi_Y$ in (\ref{bigd}) is surjective there exists $\alpha \in H^1_{\mathrm{dR}}(Y_K, (\mathcal{E}, \nabla)_K)$ such that 
$\pi_Y(\alpha)=(\omega_v)_{v\in \mathscr{V}}.$\\
Now $\pi_X([\omega]-\varphi^*(\alpha))=0$, hence, looking again at diagram (\ref{bigd}), there exists 
$c\in H^1(Gr(X_k), \mathcal{E}_K)$ such that $[\omega]-\varphi^*(\alpha)=\gamma(c)$. By the commutativity of diagram 
(\ref{bigd}) there exists an element $\mu \in H^1_{\mathrm{dR}}(Y_K, (\mathcal{E}, \nabla)_K)$ such that 
$\varphi^*(\mu)=\gamma(c).$ (One can choose $\mu=\delta(c)$.) \\
Viceversa if $[\omega]=\varphi^*(\alpha)$ for $\alpha\in H^1_{\mathrm{dR}}(Y_K, (\mathcal{E}, \nabla)_K),$ then $(\omega_v)_{v\in \mathscr{V}}=\varphi^*(\pi_Y(\alpha)):=\varphi^*(\alpha_v)_{v\in\mathscr{V}}.$ Hence $Res_{|X_e}(\omega_v)=Res_{|X_e}(\varphi^*(\alpha_v))$ for every $v\in \mathscr{V}.$ But as in the proof of lemma \ref{Nphi} one can prove that from this it follows 
that  $Res_{|X_e}(\omega_v)=0$ for every $v\in \mathscr{V}$.
\end{proof}

\section{The constant coefficients case}

In this paragraph we show that if $E$ is the trivial convergent $F$-isocrystal, then the condition in 
lemma \ref{suffnec} is fulfilled.  This will imply  that the sequence in (\ref{BC}) is exact and it will give a 
new proof of theorem \ref{teoBC} i.e. the exactness of the invariant cycles sequence under the assumption that $k$ 
is perfect instead of finite. The realization 
of $E$ on $X_K^{\rm rig}$ is the structure sheaf with trivial connection $(\mathcal{O}_{X_K}, d)$.   \\

We'd like to prove that if $[\omega]\in H^1_{\mathrm{dR}}(X^{\rm rig}_K)$ is such that 
$N([\omega])=0$, then 
one can find a hypercocycle $(\omega_v, f_e)$ representing it 
such that $\mathrm{Res}_{X_e}(\omega_{v|X_e})=0$: hence we may apply 
lemma \ref{suffnec} and  conclude.\\
Let $(\omega_v, f_e)$ be a hypercocycle representing $[\omega]$ and consider 
$\mathrm{Res}_{X_e}(\omega_{v|X_e})$; if $[\omega]$ is in $\mathrm{Ker}(N)$, then 
$(\mathrm{Res}_{X_e}(\omega_{v|X_e}))_e=0$ in $H^1(Gr(X_k), \mathcal{O}_K)$, that means that  
$(\mathrm{Res}_{X_e}(\omega_{v|X_e}))_e \in \mathrm{CoKer}(\oplus_{v\in \mathscr{V}}H^0_{\mathrm{dR}}(X_v)\rightarrow \oplus_{e\in\mathscr{E}}H^0_{\mathrm{dR}}(X_e))$.\\
On the other hand, thanks to the residue theorem on wide opens (proposition 4.3 of \cite{Co89}), for every 
irreducible component $C_v$ in $X_k,$ the family $(\mathrm{Res}_{X_e}(\omega_{v|X_e}))_e$ verifies that
\begin{equation}\label{sumres}
\sum_{e\in \mathscr{E}_v}\mathrm{Res}_{X_e}(\omega_{v|X_e})=0,
\end{equation}
where the notation $\mathscr{E}_v$ refers to the set 
$\{e \mbox{ such that there exists a vertex }w \mbox{ with } e=[v,w]\}.$ \\
Hence to prove that $\mathrm{Res}_{X_e}(\omega_{v|X_e})=0$ we are left to prove that if  
$(\mathrm{Res}_{X_e}(\omega_{v|X_e}))_e \in \mathrm{CoKer}(\oplus_{v\in \mathscr{V}}H^0_{\mathrm{dR}}(X_v)
\rightarrow \oplus_{e\in\mathscr{E}}H^0_{\mathrm{dR}}(X_e))$ and for every $v$ it verifies that 
$\sum_{e\in \mathscr{E}_v}\mathrm{Res}_{X_e}(\omega_{v|X_e})=0$, then $(\mathrm{Res}_{X_e}(\omega_{v|X_e}))_e=0$ for all $e$. 
So we are reduced to a linear algebra and graph theory problem, which  we can translate as follows.\\
Let $\mathbb{F}$ be a field of characteristic $0$.
Let $G$ be a connected graph with $n$ vertices and $m$ edges. Let us denote by $\mathscr{V}$ the set of all vertices 
and by $\mathscr{E}$ the set of all oriented edges. We use the notation $e=[v,w]$ to indicate an edge between the 
vertex $v$ and the vertex $w$.
We associate to $G$ a vector space $V=\oplus_{e\in \mathcal{E}}\mathbb{F} $ with the convention that if $e=[v,w]$ and $\bar{e}=[w,v]$ 
then $a_{e}=-a_{\bar{e}}$. Then there is a map 
$$\phi:\oplus_{v\in \mathscr{V}}\mathbb{F}\rightarrow \oplus_{e\in \mathscr{E}}\mathbb{F}  $$
$$(a_{v})_{v \in \mathscr{V}}\mapsto (a_e)_{e\in \mathscr{E}}$$
where $a_{e}=a_{v}-a_w$ if $e=[v,w]$.
We consider two vector subspaces $W$ and $T$ of $\oplus_{e\in \mathscr{E} }\mathbb{F}$ where 
$$W=\{(a_e)_{e \in \mathscr{E}} |  (a_e)_{e\in \mathscr{E}}\in \mathrm{Im}(\phi) \}$$  
$$T=\{(a_e)_{e \in \mathscr{E}} | \forall v \in \mathscr{V} \sum_{e\in \mathscr{E}_v}a_e=0 \}.$$

\begin{prop}\label{WT}
With notations as before we have 
$W\cap T=0$
\end{prop}
\begin{proof}
An element $(a_e)_{e\in\mathscr{E}}$ which belongs to $W$ and to  $T$ is described by the following equations 
$$a_e=a_v-a_w$$
$$\forall v \in \mathscr{V}\;\; \sum_{e\in \mathscr{E}_v}a_e=0.$$
We can rewrite the equations as follows:
\begin{equation}\label{deg}
\forall v \in \mathscr{V} \;\;\mathrm{deg}(v)a_v=a_{w_1}+\dots +a_{w_{s_v}}
\end{equation}
where $w_1, \dots, w_{s_v}$ are the vertices connected to $v$ by an edge and by $\mathrm{deg}(v):=s_v$ we denote the 
cardinality of the set of the vertices connected to $v$. Requiring that $W\cap T=0$ is equivalent to requiring that 
the linear system in (\ref{deg}) has a $1$-dimensional space of solutions, generated by the vector $(1,\dots, 1)$. This 
is equivalent to requiring that the matrix associated to the system in (\ref{deg}) has rank $n-1$, i.e. that there 
exists at least one minor of rank $n-1$ whose  determinant is non-zero.\\ 
This last condition is independent of the field $\mathbb{F}$, hence to prove that $W\cap T=0$ it is enough to prove that the 
equations in (\ref{deg}) imply that  $a_{v}=a_{w_i}$ for all $w_i$ and for all $v$ assuming that $\mathbb{F}$ is a totally 
ordered field.
We assume in what follows that $\mathbb{F}$ is a totally ordered field of characteristic $0$. 
Let us suppose by absurd that the equations in (\ref{deg}) do not imply that $a_{v}=a_{w}$ for all $w$. 
Let us call 
$$a_{v_0}=\mathrm{min}_{v\in \mathscr{V}}a_v$$
which exists because our assumption that our field $\mathbb{F}$ is totally ordered; then $a_{v_0}\leq a_{v}$ for all 
$v$ $\in$ $\mathscr{V}$. If $a_{v_0}= a_{v}$ for all $v$ $\in$ $\mathscr{V}$  we are done, if not there exists 
$v_1$ such that $a_{v_0}<a_{v_1}$. Moreover we can suppose that $v_1$ is connected to $v_0$ by an edge because if not, 
then this means that $a_{v_0}=a_{v}$ for all $v$ connected to $v_0$ by an edge. Then if we now fix a $v\neq v_0$ 
that is connected to $v_0$, we can consider all the $w$ that are connected to it by an edge; 
if $a_{v}=a_{w}$ for all these $w$ we can go on as before. In the end we will find that all the $a_v$ are equal for 
all  $v$ $\in$ $\mathscr{V}$ which proves the claim. \\
Hence we suppose that there exists $v_1$ such that $a_{v_0}<a_{v_1}$ for $v_1$ connected to $v_0$ by an edge. We consider 
the equation (\ref{deg}) for $v=v_0$ and we get the contradiction  
$$\mathrm{deg}(v_0)a_{v_0}<a_{w_1}+\dots a_{w_{s_v}}.$$ 
\end{proof}

With this proposition we end the proof of the exactness of the invariant cycles sequence for trivial coefficients.

\begin{rmk}
 {\rm We'd like now to give another proof of proposition \ref{WT} more in the spirit of graph theory: it uses  
proposition 4.3, proposition 
4.8 of \cite{Bi}, and lemma 13.1.1 of \cite{GoRo}.}

\begin{proof}
The matrix associated to the linear system in (\ref{deg}) is an $n\times n$ matrix $A=(a_{i,j})$, 
where for $i\neq j$ $a_{i,j}=-1h_{i,j}$ if there are $h_{i,j}$ edges between the vertex $v_i$ and $v_j$ and $0$ otherwise, 
and $a_{i,i}=\mathrm{deg}(v_i).$
We will prove that the rank of the matrix $A$ is $n-1$.\\
The matrix $A$ is called the Laplacian matrix associated to the graph $G$; we will see that $A=DD^{t}$ and that 
$D$ is an $n\times m $ matrix with rank $n-1$.\\
The following are equivalent:
\begin{itemize}
\item (i) there exists an $(n-1)\times(n-1)$ minor of $A$ with determinant different from zero,  
\item (ii) the rank of $A$ is $(n-1)$ dimensional,
\item (iii) the dimension of the $\mathrm{Kernel}$ of $A$ is $1$,
\item (iv) $\mathrm{Kernel}(D^t)=\mathrm{Kernel}(A).$
 
\end{itemize}
Assertion (i) is independent from the field $\mathbb{F}$,  so we can suppose that $\mathbb{F}$ is the field $\mathbb{R}$.\\
We will prove assertion (iv).\\
Let us suppose that $z$ is a vector in $\mathbb{R}^n$ that is in the  $\mathrm{Kernel}(A)$, we want to prove that 
$z$ $\in$ $\mathrm{Kernel}(D^t)$. Being $z$ $\in$ $\mathrm{Kernel} (A)$, then 
$$Az=0,$$
$$DD^{t}z=0$$
$$z^{t}DD^{t}z=0.$$
But the last equality implies that the vector $D^{t}z$ has inner product with itself in $\mathbb{R}^n$ equal to zero, 
that means that $D^{t}z$ is the zero vector, i.e. $z$ $\in$ $\mathrm{Kernel}(D^t), $ as we wanted.\\
We are left to prove that  $A=DD^{t}$ and that $D$ is an $n\times m $ matrix with rank $n-1$.\\
We consider the matrix $D$ associated to the graph $G$ defined as follows: $D$ is $n\times m$ matrix such that  
$(D)_{i,j}=1$ if the vertex $v_i$ is such that $e_j=[v_i,-]$, $(D)_{i,j}=-1$ if the vertex $v_i$ is such that 
$e_j=[-,v_i]$, and $(D)_{i,j}=0$ otherwise.\\ 
Now if we consider $(DD^{t})_{i,j}$, this is the inner product of the rows $\textbf{d}_i$ and $\textbf{d}_j$. They have 
a non zero entry in the same column if and only if there is an edge between $v_i$ and $v_j$, and these entries are 
one $-1$ and one $+1$, hence $(DD^{t})_{i,j}$ is given by $-1$ times the number of edges between $v_i$ and $v_j$. 
Moreover $(DD^{t})_{i,i}$ is the number of entries in $\textbf{d}_i$ different from zero, which means the degree of $v_i$. 
This proves that $A=DD^{t}$.\\
Let us see now that $D$ has rank $n-1$.\\
On every column there is a $+1$ and a $-1$, hence the sum of all the elements on the columns are zero, hence the rank of 
$D$ is less or equal to $n-1$. Let us suppose to have a linear relation
\begin{equation}\label{sum} 
\sum_ia_i\textbf{d}_i=0,
\end{equation}
where as before $\textbf{d}_i$ is the row corresponding to the vertex $v_i$ and suppose that not all the $a_i$ are zero. Choose a row $\textbf{d}_k$ for which $a_k\neq 0$. This row has non zero entries in the columns corresponding to the edges that intersect $v_i$. For every such column there is only on other row $\textbf{d}_l$ with a non zero entry in that column. Hence we should have that $a_l=a_k$, hence $a_l=a_k$ for all vertices $v_l$ adjacent to $v_k$. Hence all the $a_k$ are equal, being the graph $G$ connected, and the equation in (\ref{sum}) is a multiple of $\sum_i\textbf{d}_i=0.$ But $(a_1,\dots,a_n)$ that verifies (\ref{sum}) is in $\textrm{Kernel}(D^t)$, hence we have proven that $\textrm{Kernel}(D^t)$ is $1$-dimensional and generated by $(1,\dots,1)$, the rank is $(n-1)$-dimensional and as well as the rank of $D$.
    
\end{proof}

\end{rmk}

\section{Unipotent coefficients}
In this section we study the sequence in (\ref{BC}) when the coefficients are unipotent $F$-isocrystals. In particular we 
prove that, 
unlike the case of constant coefficients, the sequence in (\ref{BC}) is not necessarily exact. We give a sufficient 
condition for 
non exactness.\\

Let $E$ be a unipotent convergent $F$-isocrystal for which the sequence in (\ref{BC}) is exact and  let us consider the 
following extension in the category of convergent $F$-isocrystals
\begin{equation}\label{exactsequenceiso}
0\rightarrow E\stackrel{\alpha}{\rightarrow} F \stackrel{\beta}{\rightarrow} \mathcal{O}\rightarrow 0
\end{equation}
where $\mathcal{O}$ is the trivial $F$-isocrystal. 
Let us also consider the element $x \in H^1_{\mathrm{rig}}(X_k, E)$ corresponding to the class of this extension 
($x$ is then fixed by the Frobenius operator; see propositions 1.3.1   and 3.2.1 of \cite{ChLeS}) 
Let us suppose that $x\neq 0$.\\
In the sequel we use sequence ($\ref{BC}$) for the isocrystals $E$, $F$ and $\mathcal{O}$; to avoid confusion we 
denote the first maps by 
$\phi^*_{\mathcal{E}}$, $\phi^*_{\mathcal{F}}$ and $\phi^*_{\mathcal{O}}$ respectively and the monodromy operators by 
$N_{\mathcal{E}}$, $N_{\mathcal{F}}$ and $N_{\mathcal{O}_X}$ respectively.

Our assumptions imply that $H^1_{\rm rig}(X_k, E)\otimes K$ is isomorphic via $\varphi^\ast_{\mathcal{E} }$
to  ${\rm Ker}(N_{\mathcal{E}})$, and this last group contains the image of $N_{\mathcal{E}}$, as this operator has square zero.
 
\begin{teo} \label{VIP}
If $\varphi_{\mathcal{E}}^*(x\otimes 1)=N_{\mathcal{E}}(y)$ for $y\in H^1_{\mathrm{dR}}(X_K, (\mathcal{E}, \nabla)_K),$ then if we denote by 
$\alpha_{\mathrm{dR}}:H^1_{\mathrm{dR}}(X_K, (\mathcal{E}, \nabla)_K)\rightarrow H^1_{\mathrm{dR}}(X_K, (\mathcal{E}, \nabla)_K)$ the 
map induced by $\alpha$ in the  sequence \label{exactsequence}, the following holds:
$$\mathrm{Kernel} (N_{\mathcal{F}})=\bigl(H^1_{\mathrm{rig}}(X_k, F)\otimes K\bigr)\oplus K \alpha_{\mathrm{log-crys}}(y)$$

\end{teo}
\begin{proof}
Let us consider the following commutative diagram
\begin{equation}\label{bigbigbig}
\xymatrix{
H^0_{\mathrm{rig}}(X_k)\otimes K\ar[r]^(0.4){i^0_{\mathcal{O}}}\ar[d]^{\delta^0_{\mathrm{rig}}}& H^0_{\mathrm{dR}}(X_K)\ar[r]^{N^0_{\mathcal{O}}}\ar[d]^{\delta^0_{\mathrm{log-crys}}}& H^0_{\mathrm{dR}}(X_K)\ar[d]^{\delta^0_{\mathrm{dR}}}\\
H^1_{\mathrm{rig}}(X_k, E)\otimes K\ar[r]^(0.4){\varphi^*_{\mathcal{E}}}\ar[d]^{\alpha_{\mathrm{rig}}}&H^1_{\mathrm{dR}}(X_K, (\mathcal{E},
\nabla)_K)\ar[r]^{{N_{\mathcal{E}}}}\ar[d]^{\alpha_{\mathrm{dR}}}&H^1_{\mathrm{dR}}(X_K, (\mathcal{E},
\nabla)_K)       \ar[d]^{\alpha_{\mathrm{dR}}}\\
H^1_{\mathrm{rig}}(X_k, F)\otimes K\ar[r]^(0.4){\varphi^*_{\mathcal{F}}}\ar[d]^{\beta_{\mathrm{rig}}}&H^1_{\mathrm{dR}}(X_K, (\mathcal{F},
\nabla)_K)\ar[r]^{{N_{\mathcal{F}}}}\ar[d]^{\beta_{\mathrm{dR}}}&H^1_{\mathrm{dR}}(X_K, (\mathcal{F},
\nabla)_K)\ar[d]^{\beta_{\mathrm{dR}}}\\
H^1_{\mathrm{rig}}(X_k)\otimes K\ar[r]^(0.4){\varphi^*_{\mathcal{O}}} \ar[d]^{\gamma_{\mathrm{rig}}}&
H^1_{\mathrm{dR}}(X_K)\ar[r]^{{N_{\mathcal{O}}}}\ar[d]^{\gamma_{\mathrm{dR}}}&H^1_{\mathrm{dR}}(X_K)\otimes K \ar[d]^{\gamma_{\mathrm{dR}}}\\
H^2_{\mathrm{rig}}(X_k, E)\otimes K\ar[r]^(0.4){i^2_{\mathcal{E}}}\ar[d]&H^2_{\mathrm{dR}}(X_K, (\mathcal{E},
\nabla)_K) \ar[r]^{N^2_{\mathcal{E}}}\ar[d]&H^2_{\mathrm{dR}}(X_K, (\mathcal{E},
\nabla)_K)\ar[d]\\
H^2_{\mathrm{rig}}(X_k, F)\otimes K\ar[r]\ar[d]&H^2_{\mathrm{dR}}(X_K, (\mathcal{F},
\nabla)_K)\ar[r]\ar[d]&H^2_{\mathrm{dR}}(X_K, (\mathcal{F},
\nabla)_K)\ar[d]\\
H^2_{\mathrm{rig}}(X_k)\otimes K\ar[r]&H^2_{\mathrm{dR}}(X_K)\ar[r]&H^2_{\mathrm{dR}}(X_K).}
\end{equation}
Let $\varphi_{\mathcal{E}}^*(x\otimes 1)\in \varphi_{\mathcal{E}}^*(H^1_{\mathrm{rig}}(X_k, E)\otimes K)=
\mathrm{Ker}(N_{\mathcal{E}})$, with $\varphi_{\mathcal{E}}^*(x\otimes 1)=N_{\mathcal{E}}(y)$ and 
$y\in H^1_{{\rm dR}}(X_K, 
(\mathcal{E}, \nabla)_K)$. One can notice that the class of $1$ in $H^0_{\mathrm{rig}}(X_k)\otimes K=K$ is sent 
to $x\otimes 1$ in  $H^1_{\mathrm{rig}}(X_k, E)\otimes K$ by the map $\delta^0_{\mathrm{rig}}$. \\
Let us prove first that $N_{\mathcal{F}}(\alpha_{\mathrm{dR}}(y))=0$. \\
By the commutativity of the diagram (\ref{bigbigbig}) we have that  
$$N_{\mathcal{F}}(\alpha_{\mathrm{dR}}(y))=\alpha_{\mathrm{dR}}(N_{\mathcal{E}}(y))=
\alpha_{\mathrm{dR}}(\varphi_{\mathcal{E}}^*(x\otimes 1))=\alpha_{\mathrm{dR}}(\delta^0_{\mathrm{dR}}(1))=0,$$
hence $\alpha_{\mathrm{dR}}(y)\in \mathrm{Ker}(N_{\mathcal{F}}).$ \\
We claim that $z=\alpha_{\mathrm{dR}}(y)\notin \varphi^*_{\mathcal{F}}\bigl(H^1_{\mathrm{rig}}(X_k, F)\otimes K\bigr).$
Let us suppose that $z=\alpha_{\mathrm{dR}}(y)=\varphi^*_{\mathcal{F}}(b)$, with $b\in H^1_{\mathrm{rig}}(X_k, F)\otimes K$, 
then 
$$\varphi^*_{\mathcal{O}}(\beta_{\mathrm{rig}}(b))=\beta_{\mathrm{dR}}(\varphi^*_{\mathcal{F}}(b))=\beta_{\mathrm{dR}}(z)=
\beta_{\mathrm{dR}}(\alpha_{\mathrm{dR}}(y))=0.$$
As $\varphi^*_{\mathcal{O}}$ is injective we have $\beta_{\mathrm{rig}}(b)=0$, hence $b \in \mathrm{Ker} (\beta_{\mathrm{rig}})=
\mathrm{Im} (\alpha_{\mathrm{rig}})$, i.e. there exists $a\in H^1_{\mathrm{rig}}(X_k, E)\otimes K$ such that 
$\alpha_{\mathrm{rig}}(a)=b.$ So 
$$z=\alpha_{\mathrm{dR}}(y)=\varphi^*_{\mathcal{F}}(b)=\varphi^*_{\mathcal{F}}(\alpha_{\mathrm{rig}}(a))=\alpha_{\mathrm{dR}}
(\varphi^*_{\mathcal{E}}(a)),$$
from which it follows that 
$$y-\varphi^*_{\mathcal{E}}(a) \in \mathrm{Ker}( \alpha_{\mathrm{dR}})=\mathrm{Im} (\delta^0_{\mathrm{dR}}).$$
But the image of $\delta^0_{\mathrm{dR}}$ is generated by $\varphi^*_{\mathcal{E}}(x\otimes 1)$, as vector space, 
hence $y-\varphi^*_{\mathcal{E}}(a)=m\varphi^*_{\mathcal{E}}(x\otimes 1)$ for some $m\in K$.\\
Now 
$$N_{\mathcal{E}}(y)-N_{\mathcal{E}}(\varphi^*_{\mathcal{E}}(a))=N_{\mathcal{E}}(m\varphi^*_{\mathcal{E}}(x\otimes 1))=0,$$
hence
$$N_{\mathcal{E}}(y)=N_{\mathcal{E}}(\varphi^*_{\mathcal{E}}(a))=0,$$
but 
$$N_{\mathcal{E}}(y)=\varphi^*_{\mathcal{E}}(x\otimes 1)=0,$$
which is absurd.\\
We are left to prove that $\forall \alpha \in \mathrm{Ker} N_{\mathcal{F}}$ there exists 
$\beta \in H^1_{\mathrm{rig}}(X_k, F)\otimes K$ 
and $t\in K$ such that $\alpha=\varphi^*_{\mathcal{F}}(\beta)+t\alpha_{\mathrm{dR}}(y).$ 
Let us calculate 
$$N_{\mathcal{O}}(\beta_{\mathrm{dR}}(\alpha))=\beta_{\mathrm{dR}}(N_{\mathcal{F}}(\alpha)))=0,$$
hence 
$$\beta_{\mathrm{dR}}(\alpha) \in \mathrm{Ker} (N_{\mathcal{O}})=\mathrm{Im} (\varphi^*_{\mathcal{O}}),$$
so that there exists $\gamma \in H^1_{\mathrm{rig}}(X_k)\otimes K$ such that $\varphi^*_{\mathcal{O}}(\gamma)=
\beta_{\mathrm{dR}}(\alpha)$. 
By lemma \ref{zeromap} we have $\gamma_{\mathrm{rig}}(\gamma)=0$. Hence there exists 
$\beta \in H^1_{\mathrm{rig}}(X_k, F)\otimes K$ such 
that $\beta_{\mathrm{rig}}(\beta)=\gamma.$
Let us consider now the element $\alpha-\varphi^*_{\mathcal{F}}(\beta);$ it is in the $\mathrm{Kernel}$ of 
$\beta_{\mathrm{dR}},$ because
$$\beta_{\mathrm{dR}}(\alpha-\varphi^*_{\mathcal{F}}(\beta))=\beta_{\mathrm{dR}}(\alpha)-\varphi^*_{\mathcal{O}}(\beta_{\mathrm{rig}}(b))=
\beta_{\mathrm{dR}}(\alpha)-\varphi^*_{\mathcal{O}}(\gamma)=0.$$
Hence there exists $u\in H^1_{\mathrm{dR}}(X_K, (\mathcal{E},\nabla_\mathcal{E})_K)$ such that 
$\alpha_{\mathrm{dR}}(u)=\alpha-\varphi^*_{\mathcal{F}}(\beta).$ Now 
$$\alpha_{\mathrm{dR}}(N_{\mathcal{E}}(u))=N_{\mathcal{F}}(\alpha_{\mathrm{dR}}(u))=N_{\mathcal{F}}(\alpha-\varphi^*_{\mathcal{F}}
(\beta))=0$$
because $\alpha \in \mathrm{Ker}( N_{\mathcal{F}})$ and $N_{\mathcal{F}}(\varphi^*_{\mathcal{F}}(\beta))=0$ 
by lemma \ref{Nphi}. Then $N_{\mathcal{E}}(u)\in \mathrm{Ker} (\alpha_{\mathrm{dR}})={\rm Im}\, \delta^0_{\mathrm{dR}}$, i.e 
$N_{\mathcal{E}}(u)=t\varphi^*_{\mathcal{E}}(x\otimes 1)=tN_{\mathcal{E}}(y)$, for some $t\in K$ and $u-ty \in \mathrm{Ker} (N_{\mathcal{E}})=\varphi^*_{\mathcal{E}}(H^1_{\mathrm{rig}}(X_k, E)\otimes K)$. Hence there exists $\beta' \in H^1_{\mathrm{rig}}(X_k, E)\otimes K$ 
such that $u=ty+\varphi^*_{\mathcal{E}}(\beta').$ So 
$$\alpha-\varphi^*_{\mathcal{F}}(\beta)=\alpha_{\mathrm{dR}}(u)=\alpha_{\mathrm{dR}}(ty+\varphi^*_{\mathcal{E}}(\beta'))=
t\alpha_{\mathrm{dR}}(y)+\alpha_{\mathrm{dR}}(\varphi^*_{\mathcal{E}}(\beta'))$$
which means that
$$\alpha=\varphi^*_{\mathcal{F}}(\beta)+t\alpha_{\mathrm{dR}}(y)+\alpha_{\mathrm{dR}}(\varphi^*_{\mathcal{E}}(\beta')),$$
but  $\varphi^*_{\mathcal{F}}(\beta)+\alpha_{\mathrm{dR}}(\varphi^*_{\mathcal{E}}(\beta'))=\varphi^*_{\mathcal{F}}(\beta)+
\varphi^*_{\mathcal{F}}(\alpha_{\mathrm{rig}}(\beta')),$ hence we are done. 
\end{proof}

\begin{lemma}\label{zeromap} With the same hypothesis and notations as in the previous theorem, the co-boundary map 
$\gamma_{\mathrm{rig}}: H^1_{\rm rig}(X_k)\otimes K\longrightarrow H^2_{\rm rig}(X_k, E)\otimes K$ 
induced by the exact sequence (\ref{exactsequenceiso}) is the zero map.
\end{lemma}
\begin{proof}

Clearly, the vanishing of $\gamma_{\mathrm{rig}}$ is equivalent to the fact that the map
$j:H^2_{\rm rig}(X_k, E)\otimes K\lra H^2_{\rm rig}(X_k, F)\otimes K$ is {\bf injective}.

Let us first make more explicit the group $H^2_{\rm rig}(X_k, G)\otimes K$, where $G$ is any one of
the isocrystals $E, F, \mathcal{O}$ and $(\mathcal{G}, \nabla)$ is the module with integrable connection that it induces. 
Let us 
recall the notations of section \ref{Monodromy operator and rigid cohomology}:
we consider the diagram
$$
X_k\hookrightarrow P_k\stackrel{sp_{P_{\cal V}}}{\longleftarrow} P_K
$$
with $P_k$ smooth and let $Y_K:=sp_{P_{\cal V}}^{-1}(X_k)$. Then $H^i_{\rm rig}(X_k, G)\otimes K=
H^i_{\rm dR}(Y_K, (\mathcal{G}, \nabla)_K)$.

The relevant part of the Mayer-Vietoris exact sequence for the admissible 
covering $\{Y_v\}_v$ of $Y_K$ then reads
$$
\oplus_e H^1_{\rm dR}(Y_e, (\mathcal{G}, \nabla)_K)\longrightarrow H^2_{\rm dR}(Y_K,(\mathcal{G}, \nabla)_K) \longrightarrow 
\oplus_v H^2_{\rm dR}(Y_v, (\mathcal{G}, \nabla)_K)
\longrightarrow \oplus_e H^2_{\rm dR}(Y_e,(\mathcal{G}, \nabla)_K).
$$
As $Y_e$ is a wide open polydisk, $H^i_{\rm dR}(Y_e, (\mathcal{G}, \nabla)_K)=0$ for $i\ge 1$, therefore we have a 
natural isomorphism $H^2_{\rm dR}(Y_K, (\mathcal{G}, \nabla)_K) \cong \oplus_v H^2_{\rm dR}(Y_v, (\mathcal{G}, \nabla)_K).$ \\
Moreover, as $C_v$ which is the irreducible component of $X_k$ corresponding to 
$v$ was supposed smooth it follows that we have canonical isomorphisms
$H^i_{\rm dR}(Y_v, (\mathcal{G}, \nabla)_K)\cong H^i_{\rm crys}(C_{v}, G)\otimes K$. In particular,
if we denote by $Z_v$ a smooth proper curve over $K$ whose reduction is $C_{v}$ and
which contains the wide open $X_v$, then the isocrystal $G$ can be evaluated on $Z_v$ to give a 
sheaf with connection which we'll denote again by $(\mathcal{G}, \nabla)$. Then
$H^i_{\rm dR}(Y_v, (\mathcal{G}, \nabla)_K)\cong H^i_{\rm dR}(Z_v, (\mathcal{G}, \nabla)_K)$ for all $i\ge 0$.

Therefore we have a natural isomorphism $H^2_{\rm rig}(X_k, G)\otimes K\cong 
\oplus_v H^2_{\rm dR}(Z_v, (\mathcal{G}, \nabla)_K)$.

For every vertex $v$ we denote as before by $\mathscr{E}_v:=\{e \mbox{ such that there exists a vertex }w 
\mbox{ with } e=[v,w]\}$. 
For every $v$ and $e\in \mathscr{E}_v$ we denote by $D_e$ the residue disk  of the point in $C_v$
corresponding to $e$ in $Z_v$. Let us then remark that the family $\{ X_v, D_e\}_{e\in \mathscr{E}_v}$ is an admissible
covering of $Z_v$ and $X_v\cap D_e=X_e$ for every $e\in \mathscr{E}_v$. We will represent classes in
$H^2_{\rm dR}(Z_v, (\mathcal{G}, \nabla)_K)$ by hypercocycles for the above covering.\\

We now prove the injectivity of $j:H^2_{\rm rig}(X_k, E)\otimes K\lra H^2_{\rm rig}(X_k, F)\otimes K$.\\
 Let
$z\in H^2_{\mathrm{rig}}(X_k, E)\otimes K=\oplus_{v}H^2_{\rm dR}(Z_v, (\mathcal{E}, \nabla)_K)$ such that  
$j(z)=0$.  Let
$z_v\in H^2_{\rm dR}(Z_v, (\mathcal{E}, \nabla)_K)$ be the $v$-component of $z$ and 
$j_v:H^2_{\rm dR}(Z_v, (\mathcal{E}, \nabla)_K)\lra H^2_{\mathrm{dR}}(Z_v, (\mathcal{F}, \nabla)_K)$
be the $v$ component of $j$. Obviously $j_v(z_v)=0$ and it would be enough to show that this implies $z_v=0$ for every
$v$.

Let $\bigl(\omega_e  \bigr)_{e\in \mathcal{E}_v}$ be a $2$-hyper cocycle representing $z_v$, where
$\omega_e\in H^0(X_e, \mathcal{E}_K\otimes\Omega^1_{Z_v})$ for all $e$. Then $j_v(z_v)$ will be represented by
the $2$-hyper cocycle $\bigl(\alpha(\omega_e)  \bigr)_{e\in \mathscr{E}_v}$, where let us recall
$\alpha$ is defined by the exact sequence of isocrystals on $X_k$ below
$$
0\lra \mathcal{E}\stackrel{\alpha}{\lra}\mathcal{F}\stackrel{\beta}{\lra}\mathcal{O}\lra 0.
$$  
As extension on $X_K$ this is given by the class $\varphi^*_{\mathcal{E}}(x\otimes 1)=N_{\mathcal{E}}(y)\in H^1_{\rm dR}(X_K, 
(\mathcal{E}, \nabla)_K)$
and therefore, for every $v$, the sequence
$$
0\lra H^0(X_v, \mathcal{E}_K) \stackrel{\alpha}{\lra} H^0(X_v, \mathcal{F}_K) \stackrel{\beta}\lra H^0(X_v, 
\mathcal{O}_{X_K})\lra 0,
$$
is exact because $X_v$ are wide opens and moreover, it is naturally split as an exact sequence of $\cO_{X_v}$-modules with 
connections because $\varphi_{\cal E}^*(x\otimes 1)=N_{\mathcal{E}}(y)$ can be represented by $(0_v, f_e)$ with 
$f_e\in H^0_{\mathrm{dR}}(X_e, (\mathcal{E}, \nabla)_K)$. Let $s:H^0(X_v, \mathcal{O}_{X_K})\lra H^0(X_v,\mathcal{ F}_K)$ 
be such a section of
$\beta$. We remark that it is determined by $s(1)$, which is an element of $H^0_{\rm dR}(X_v, (\mathcal{F}, \nabla)_K)$
such that $\beta(s(1))=1$. 

Therefore, $s$ determines, for every $e\in \mathscr{E}_v$, a splitting of the exact sequence 
$$
0\lra H^0(X_e, \mathcal{E}_K)\stackrel{\alpha_e}{\lra} H^0(X_e, \mathcal{F}_K)\stackrel{\beta_e}{\lra} H^0(X_e, 
\mathcal{O}_{X_K})\lra 0
$$
which will also be called $s_e$ (it is determined by the element $s_e(1)=s(1)|_{X_e})$.

Now the sequence 
\begin{equation}\label{exactone}
\quad 0\lra H^0_{\rm dR}(X_e, (\mathcal{E}, \nabla)_K)\stackrel{\alpha_e}{\lra} H^0_{\rm dR}(X_e, (\mathcal{F}, \nabla)_K)
\stackrel{\beta_e}{\lra} 
H^0_{\rm dR}(X_e, (\mathcal{O}_{X_K}, d))\lra 0
\end{equation}
is exact and $s_e$ induces a natural splitting of it.

The isocrystal $G$ (which is any one of $E, F, \mathcal{O}$ regarded as 
a sheaf with connection on $Z_v$)
has a basis of horizontal sections on $D_e$, for every $e\in \mathscr{E}_v$. Therefore 
the natural restriction map $H^0_{\rm dR}(D_e, (\mathcal{G}, \nabla)_K)\lra H^0_{\rm dR}(X_e, (\mathcal{G}, \nabla)_K)$
is an isomorphism.
Thus  the exact sequence (\ref{exactone}) implies that the sequence
\begin{equation}\label{exactone2}
\quad 0\lra H^0_{\rm dR}(D_e, (\mathcal{E}, \nabla)_K)\stackrel{\alpha_e}{\lra} H^0_{\rm dR}(D_e, (\mathcal{F}, \nabla)_K)
\stackrel{\beta_e}{\lra} 
H^0_{\rm dR}(D_e, (\mathcal{O}_{X_K}, d))\lra 0
\end{equation}
is exact and naturally split, where we denote the splitting by $s_e$. By tensoring (\ref{exactone2})
with $\Omega^1_{D_e}$ we obtain that the
sequence
$$
0\lra H^0(D_e, \mathcal{E}_K\otimes \Omega^1_{D_e})\stackrel{\alpha_e}{\lra} H^0(D_e, \mathcal{F}_K \otimes\Omega^1_{D_e})
\stackrel{\beta_e}{\lra} 
H^0(D_e, \Omega^1_{D_e})\lra 0
$$
is exact, naturally split as sequence of $\cO_{D_e}$-modules with connection and everything is 
compatible with restriction to $X_e$.

Using these splittings, we write $H^0(X_v, \mathcal{F}_K\otimes\Omega^1_{X_v})=
H^0\bigl(X_v, \mathcal{E}_K\otimes\Omega^1_{X_v}\bigr)\oplus H^0\bigl(X_v, \Omega^1_{X_v}\bigr)$ and 
similarly for sections over $X_e$ and $D_e$.

Now we go back to proving that $j_v$ is injective for all $v$. Suppose that $j_v(z_v)=0$,
i.e. for every $e\in \mathscr{E}_v$, $\alpha_e(\omega_e)=\eta_v|_{X_e}-\rho_e|_{X_e}-\nabla(f_e)$,
where $\eta_v\in H^0\bigl(X_v, \mathcal{F}_K\otimes\Omega^1_{X_v}\bigr), \rho_e\in 
H^0\bigl(D_e, \mathcal{F}_K\otimes\Omega^1_{D_e}\bigr), f_e\in H^0\bigl(X_e, \mathcal{F}_K\bigr)$.

Using the decompositions above we write (uniquely): $\eta_v=\eta_{v,E}+\eta_{v,\cO}$, 
$\rho_e=\rho_{e,E}+\rho_{e,\cO}$ and $f_e=f_{e,E}+f_{e,\cO}$, with
$\eta_{v,E}\in H^0(X_v, \mathcal{E}_K\otimes\Omega^1_{X_v})$, $\rho_{e,E}\in H^0\bigl(D_e, \mathcal{E}_K\otimes
\Omega^1_{D_e}\bigr)$ etc.

Using the fact that the decompositions respect the connections and the restrictions to $X_e$, we obtain:
$$
\omega_e-\Bigl(\eta_{v,E}|_{X_e}-\rho_{e,E}|_{X_e}-\nabla(f_{e,E} ) \Bigr)=
\eta_{v,\cO_X}|_{X_e}-\rho_{e,\cO_X}|_{X_e}-d_X(f_{e,\cO}). 
$$
As the decomposition is a direct sum decomposition the LHS and the RHS are $0$.

Therefore $\omega_e=\eta_{v,E}|_{X_e}-\rho_{e,E}|_{X_e}-\nabla(f_{e,E})$ for every $e\in \mathscr{E}_v$
and we have $z_v$=0.
\end{proof}

\appendix
\section{An example for a Tate curve}

In this paragraph we use explicit calculations to confirm theorem \ref{VIP}, i.e.
that the  sequence (\ref{BC}) is not exact for a certain non-trivial unipotent $F$-isocrystal $E$ on 
a specific Tate curve. 

Let $X$ be a Tate elliptic curve over $K$ with invariant $q$, where 
$q \in m_\mathcal{V}=(\pi)$. 
We consider $x \in H^1_{\mathrm{rig}}(X_k)$. Thanks to what we said before $\varphi_{\cal O}^*(x\otimes 1)$ in $H^1_{\mathrm{dR}}(X_K)$ is such 
that 
$N(\varphi_{\mathcal{O}}^*(x\otimes 1))=0$; since $H^1_{\mathrm{dR}}(X_K)$ is a $2$-dimensional $K$-vector space, then  
$ \mathrm{Im}(N)=\mathrm{Ker}(N)$, hence $\varphi_{\cal O}^*(x\otimes 1)\in \mathrm{Im}(N)$. This means that in this case the 
hypothesis of the theorem 
\ref{VIP} are satisfied. \\
Every element in $H^1_{\mathrm{rig}}(X_k)$ corresponds to an extension of the trivial $F$-isocrystal by itself 
(proposition 1.3.1 of \cite{ChLeS}), hence the element $x$ corresponds to the following exact sequence
$$0\rightarrow \mathcal{O}\rightarrow E \rightarrow \mathcal{O}\rightarrow 0.$$
As before we consider $\varphi_{\cal O}^*(x\otimes 1) \in H^1_{\mathrm{dR}}(X_K)$ and the exact sequence of modules with connections induced 
by the one above:
$$0\rightarrow (\mathcal{O}_{X_K},d)\rightarrow(\mathcal{E},\nabla)_K\rightarrow (\mathcal{O}_{X_K},d)\rightarrow 0.$$
We suppose from now on that  $\mathrm{ord}_{\pi}q=3.$ 
Then the graph associate to 
$X$ is a triangle with vertices $I,II,III$  and edges $[I,II], [II,III], [I,III]$.\\

The element  $\varphi_{\cal O}^*(x\otimes 1)$, as hypercocycle, can be written as $(0_v, g_e)$ with $g_e \in H^0(X_e)$; 
in particular $d(g_e)=0$, so $g_e\in K$. Moreover since $E$ is an $F$-isocrystal, the class $x$ is 
fixed by the Frobenius of $H^1_{\mathrm{rig}}(X_k)$ (\cite{ChLeS} prop 3.2.1), in particular we can take 
$g_{e}\in \mathbb{Q}_p$ for every $e$.\\
The $\mathcal{O}_{X_K}$-module $\mathcal{E}_K$ is locally free: on $X_v$ it has a basis given by $e_{1,v}, e_{2,v}$ and 
on $X_w$ it has a basis given by $e_{1,w}, e_{2,w}.$ If on $X_e$ we choose $e_{1,v}, e_{2,v}$ as basis, then the 
changing basis matrix is given by 
\begin{displaymath} \left(\begin{array}{cc}
1 & g_e \\0&1 \end{array}\right) \end{displaymath}
and the connection on $X_e$ is given by the direct sum of the two trivial connections.\\
Now we consider $(\omega_v, f_e)\in H^1_{\mathrm{dR}}(X_K, (\mathcal{E},\nabla)_K)$, then
$$\omega_v=h_{1,v}e_{1,v}+h_{2,v}e_{2,v},$$
$$\omega_w=h_{1,w}e_{1,w}+h_{2,w}e_{2,w},$$
$$\omega_{w_{|X_e}}=(h_{1,w}+g_eh_{2,w})e_{1,v}+h_{2,w}e_{2,w},$$
with $h_{1,v}$ and $h_{2,v}$ elements of $\Omega^1_{X_v}$ and $h_{1,w}$ and $h_{2,w}$ elements of $\Omega^1_{X_w}.$ 
Let us suppose now that $(\omega_v,f_e)\in \mathrm{Kernel}(N_{\mathcal{E}}),$ which means that 
$$N_{\mathcal{E}}(\omega_v,f_e)=(0,Res_{|X_{e}}\omega_v)=0 \,\,\,\mathrm{in} \,\,\,H^1_{\mathrm{dR}}(X_K, (\mathcal{E},\nabla)_K),$$
but as the map from $ H^1(Gr, \mathcal{E}_K)$ to $H^1_{\mathrm{dR}}(X_K, (\mathcal{E},\nabla)_K)$ is injective, we have that
 $Res_{|X_{e}}\omega_v$ is zero as element of $ H^1(Gr, \mathcal{E}_K)$. 

Let us write the system which tells us that an element $a_e=(a_e^1, a^2_e) \in H^1(Gr, 
\mathcal{E}_K)=\frac{\oplus_e H^0_{\mathrm{dR}}(X_e, (\mathcal{E}, \nabla)_K)}{\oplus_vH^0_{\mathrm{dR}}(X_v, 
(\mathcal{E}, \nabla)_K)}$, written in coordinates with respect to the basis $e_{v,1}, e_{v,2}$, is zero:
\begin{displaymath}
\left\{ \begin{array}{l}
a_{[I,II]}^1=a_I^1-a_{II}^1-g_{[I,II]}a_{II}^2\\ 
a_{[I,II]}^2=a_I^2-a_{II}^2 \\ 
\end{array} \right. 
 \end{displaymath}
 \begin{displaymath}
\left\{ \begin{array}{l}
a_{[II,III]}^1=a_{II}^1-a_{III}^1-g_{[II,III]}a_{III}^2\\ 
a_{[II,III]}^2=a_{II}^2-a_{III}^2 \\ 
\end{array} \right. 
 \end{displaymath}
\begin{displaymath}
\left\{ \begin{array}{l}
a_{[I,III]}^1=a_I^1-a_{III}^1-g_{[I,III]}a_{III}^2\\ 
a_{[I,III]}^2=a_I^2-a_{III}^2 \\ 
\end{array} \right. 
 \end{displaymath}
Moreover from the Gysin sequence ( \cite{ChLeS} proposition 2.1.4), applied to every component $C_v$ of $X_k$ 
(on every wide open 
$X_v$ $(\mathcal{E},\nabla)_K$ is the direct sum of two copies of $(\mathcal{O}_{X_K},d)$), we can derive the following 
equations:
\begin{displaymath}
\left\{ \begin{array}{l}
a_{[I,II]}^1+a^1_{[I,III]}=0\\ 
a_{[II,III]}^1+a^1_{[II,I]}=0\\
a_{[III,I]}^1+a^1_{[III,II]}=0 
\end{array} \right. 
 \end{displaymath}
\begin{displaymath}
\left\{ \begin{array}{l}
a_{[I,II]}^2+a^2_{[I,III]}=0\\ 
a_{[II,III]}^2+a^2_{[II,I]}=0\\
a_{[III,I]}^2+a^2_{[III,II]}=0 
\end{array} \right. 
 \end{displaymath}
Putting together the previous equations and writing a linear system in terms of the $a_v$'s, we find the following 
matrix
 \begin{displaymath} A=\left(\begin{array}{cccccc}
2 & 0 & -1 & -g_{[I,II]} & -1& -g_{[I,III]} \\ 
0 & 2 & 0 & -1 & 0 & -1\\
-1 & -g_{[II,I]} & 2 & 0 & -1 & -g_{[II,III]}\\
0 & -1& 0 & 2 & 0 & -1\\
-1 & -g_{[III,I]} & -1 & -g_{[III, II]} & 2 & 0 \\
0 & -1 & 0 & -1 & 0 & 2
 \end{array}\right) \end{displaymath}
where $g_{[I,II]}=-g_{[II,I]}$, $g_{[II,III]}=-g_{[III,II]}$ and $g_{[I,III]}=-g_{[III,I]}.$ The matrix $A$ has determinant equal to 
zero and dimension of the rank equal to 4.
Two generators of the Kernel are the following vectors:
$$K_1=(1,0,1,0,1,0)$$
$$K_2=(\frac{1}{3}g_{[I,II]}+\frac{2}{3}g_{[I,III]}+\frac{1}{3}g_{[II,III]},1,-\frac{1}{3}g_{[I,II]}+\frac{1}{3}g_{[I,III]}+
\frac{2}{3}g_{[II,III]},1,0,1).$$
If we now write $K_1$ and $K_2$ as elements of $H^1(Gr,\mathcal{E}_K)$, i.e. as elements of 
$\oplus_e H^0_{\mathrm{dR}}(X_e, (\mathcal{E},\nabla)_K)$, we find the following vectors:
$$H_1=(0,0,0,0,0,0),$$
$$H_2=(-\frac{1}{3}g_{[I,II]}-\frac{1}{3}g_{[II,III]}+\frac{1}{3}g_{[I,III]}, 0, -\frac{1}{3}g_{[I,II]}-\frac{1}{3}g_{[II,III]}+\frac{1}{3}g_{[I,III]}, 0, \frac{1}{3}g_{[I,II]}+\frac{1}{3}g_{[II,III]}-\frac{1}{3}g_{[I,III]},0).$$
These computations show that the Kernel of $N_{\mathcal{E}}$ consists of the $(\omega_v, f_e)\in H^1_{\mathrm{dR}}(X_K, 
(\mathcal{E}, \nabla)_K)$ such that $Res_{|X_e} \omega_v$ equals $H_1$ or $H_2$. The elements 
$(\omega_v, f_e)$ of $H^1_{\mathrm{dR}}(X_K, (\mathcal{E}, \nabla)_K)$ which are such that 
$Res_{|X_e}\omega_v= H_1$ are the elements that come from $H^1_{\mathrm{rig}}(X_k, E)\otimes K$. \\
Let us consider now the subvector space 
$$V=\{(\omega_v, f_e)| Res_{|X_e}\omega_v=t H_2, \mathrm{with} \,t\in K\}$$ 
Clearly the elements of $\varphi_{\cal E}^*(H^1_{\mathrm{rig}}(X_k, E)\otimes K)$ are contained in $V$ and one can see that 
$V/\varphi_{\cal E}^*(H^1_{\mathrm{rig}}(X_k, E)\otimes K)$ is a $1$-dimensional vector space, in fact two elements in $V$ are multiples 
one of the other modulo an element of $\varphi_{\cal E}^*(H^1_{\mathrm{rig}}(X_k, E)\otimes K)$.

Dip.Matematica Univ.Padova, chiarbru@math.unipd.it,
Dept.Math. UC Berkeley coleman@math.berkeley.edu
Math. Institute, Univ.Tokyo proietto@ms.u-tokyo.ac.jp
Dept.Math. Concordia, Dip.Matematica Univ.Padova iovita@math.unipd.it

\end{document}